\documentclass{amsart}
\usepackage{amssymb}
\usepackage{graphicx}
\usepackage[all]{xy}

\newcommand{\incgr}[1]{\hbox{$\vcenter{%
  \hbox{\includegraphics%
  {#1.eps}}}$}}

\newcommand{\sst}{\scriptscriptstyle}
\newcommand{\df}[1]{\textbf{#1}}

\newcommand{\ssm}{\smallsetminus}
\newcommand{\red}{{\mathrm{red}}}

\newcommand{\bA}{\mathbb{A}}

\newcommand{\F}{\mathbb{F}}
\newcommand{\Gm}{\mathbb{G}_{\mathrm{m}}}
\newcommand{\bP}{\mathbb{P}}
\newcommand{\Q}{\mathbb{Q}}
\newcommand{\Z}{\mathbb{Z}}

\newcommand{\fu}{\mathfrak{u}}

\newcommand{\cA}{\mathcal{A}}
\newcommand{\cC}{\mathcal{C}}
\newcommand{\cD}{\mathcal{D}}
\newcommand{\cF}{\mathcal{F}}
\newcommand{\cG}{\mathcal{G}}
\newcommand{\cI}{\mathcal{I}}
\newcommand{\cL}{\mathcal{L}}
\newcommand{\cM}{\mathcal{M}}
\newcommand{\cO}{\mathcal{O}}

\newcommand{\cS}{\mathcal{S}}

\newcommand{\sP}{\mathsf{P}}

\newcommand{\vD}{\mathsf{D}}
\newcommand{\sgD}{\mathbb{D}}
\newcommand{\IC}{\mathrm{IC}}
\newcommand{\cIC}{\mathcal{IC}}

\newcommand{\cano}[1]{\Omega_{#1}^{\mathrm{top}}}

\newcommand{\Dbm}[1]{\mathrm{D}^{\mathrm{b}}_{\mathrm{m}}(#1)}
\newcommand{\Dbml}[2]{\mathrm{D}^{\le #2}}
\newcommand{\Dbmg}[2]{\mathrm{D}^{\ge #2}}
\newcommand{\Dwl}[1]{\mathrm{D}_{\le #1}}
\newcommand{\Dwg}[1]{\mathrm{D}_{\ge #1}}

\newcommand{\p}{{}^p}
\newcommand{\bru}{{}^{\bar r}}
\newcommand{\ru}{{}^{r}}
\newcommand{\q}{{}_q}
\newcommand{\cheq}{{}_{\hat q}}

\newcommand{\cg}[1]{\cC_G(#1)}
\newcommand{\cgl}[2]{\cC_G(#1)_{\le #2}}
\newcommand{\cgg}[2]{\cC_G(#1)_{\ge #2}}

\newcommand{\Db}[1]{\cD_{\sst G}^{\sst\mathrm{b}}(#1)}
\newcommand{\Dm}[1]{\cD_{\sst G}^{\sst -}(#1)}
\newcommand{\Dp}[1]{\cD_{\sst G}^{\sst +}(#1)}

\newcommand{\Dl}[2]{\Db{#1}^{\le #2}}
\newcommand{\Dg}[2]{\Db{#1}^{\ge #2}}
\newcommand{\Dml}[2]{\Dm{#1}^{\le #2}}
\newcommand{\Dpg}[2]{\Dp{#1}^{\ge #2}}

\newcommand{\Dsl}[2]{\Db{#1}_{\le #2}}
\newcommand{\Dsg}[2]{\Db{#1}_{\ge #2}}
\newcommand{\Dsp}[2]{\Db{#1}_{[#2]}}
\newcommand{\Dmsl}[2]{\Dm{#1}_{\le #2}}
\newcommand{\Dpsg}[2]{\Dp{#1}_{\ge #2}}

\newcommand{\Dkl}[2]{\Db{#1}_{\sqsubseteq #2}}
\newcommand{\Dkg}[2]{\Db{#1}_{\sqsupseteq #2}}
\newcommand{\Dmkl}[2]{\Dm{#1}_{\sqsubseteq #2}}
\newcommand{\Dpkg}[2]{\Dp{#1}_{\sqsupseteq #2}}

\newcommand{\Cento}{\mathrm{Z}^\circ}
\newcommand{\Rep}[1]{\mathfrak{Rep}(#1)}
\newcommand{\la}{\langle}
\newcommand{\ra}{\rangle}
\newcommand{\hto}{\hookrightarrow}
\newcommand{\Qlb}{\bar{\mathbb{Q}}_{\ell}}

\newcommand{\orb}[1]{\mathbb{O}(#1)}

\DeclareMathOperator{\alt}{alt}
\DeclareMathOperator{\charac}{char}
\DeclareMathOperator{\cod}{cod}
\DeclareMathOperator{\cok}{cok}
\DeclareMathOperator{\End}{End}
\DeclareMathOperator{\Ext}{Ext}
\DeclareMathOperator{\Hom}{Hom}
\DeclareMathOperator{\cHom}{\mathcal{H}\mathit{om}}
\DeclareMathOperator{\Irr}{Irr}
\DeclareMathOperator{\Lie}{Lie}

\DeclareMathOperator{\step}{step}
\DeclareMathOperator{\cRHom}{R\mathcal{H}\mathit{om}}
\DeclareMathOperator{\skdeg}{sk\,deg}
\newcommand{\Lotimes}{\mathchoice%
  {\overset{\scriptscriptstyle L}{\otimes}}%
  {\otimes^{\scriptscriptstyle L}}{\otimes^L}{\otimes^L}}

\newtheorem{thm}{Theorem}[section]
\newtheorem{prop}[thm]{Proposition}
\newtheorem{cor}[thm]{Corollary}
\theoremstyle{definition}
\newtheorem{defn}[thm]{Definition}

\newcounter{pervenum}

\numberwithin{equation}{section}

\makeatletter
\def\usectrnz#1{\@nmbrlisttrue\def\@listctr{#1}}
\makeatother
\newenvironment{eqenum}{\begin{list}{(\theequation)}{%
   \usectrnz{equation}\def\makelabel##1{{\upshape ##1}}%
   \setlength{\itemsep}{\medskipamount}\setlength{\topsep}{\medskipamount}%
   \setlength{\labelwidth}{0pt}\setlength{\labelsep}{12pt}%
   \setlength{\leftmargin}{12pt}}}{\end{list}}

\newenvironment{proofcom}{\begin{proof}[Remarks on proof]%
  }{\end{proof}}

\title{Introduction to staggered sheaves}
\author{Pramod N. Achar}

\begin{document}

\begin{abstract}
This note is an expository account of the theory of staggered sheaves, based on a series of lectures given by the author at RIMS (Kyoto) in October 2008.
\end{abstract}

\maketitle

Perverse sheaves have become a tool of great importance in representation
theory, largely because of the remarkable way in which they provide a link
between geometry and algebra.  A single perverse sheaf contains
information reminiscent of classical algebraic topology; indeed, the
prototypical example of a perverse sheaf comes from the Goresky--MacPherson
theory of ``intersection homology.''  But the category of perverse
sheaves behaves like the module categories typically seen in representation
theory, and some of the most important theorems here are Ext-vanishing and
complete reducibility criteria.

\emph{Staggered sheaves}, the subject of the present note, were introduced by the author in~\cite{a:sts} and subsequently studied in a series of papers~\cite{at:bstc, at:pdtss, a:qhpss} by the author and D.~Treumann.  The category of staggered sheaves also enjoys a long list of
remarkable algebraic properties, resembling the most important properties
of perverse sheaves.  Most notably, the category of staggered sheaves is
quasi-hereditary and exhibits ``purity'' and ``decomposition'' phenomena. 
But whereas perverse sheaves are built out of local systems, staggered
sheaves are built out of vector bundles, and they live in the derived
category of coherent sheaves.

This note is a self-contained exposition of the main results on staggered
sheaves.  It is meant to be accessible to anyone familiar with the basic
theory of algebraic groups and a passing acquaintance with derived
categories.  Some familiarity with
perverse sheaves may be helpful for motivation, but
the relevant facts will be reviewed in
Section~\ref{sect:perv}.  Most proofs will be omitted.  No new results appear here.

 Notation and definitions occupy Section~\ref{sect:prelim}, and the main
theorems are listed in Sections~\ref{sect:basic}--\ref{sect:quasi}.  An example on
$\bP^1$ is worked out in Section~\ref{sect:exam}, and some
open questions are posed in Section~\ref{sect:ques}. 
Appendix~\ref{sect:baric} surveys the theory of \emph{baric structures},
required for the proofs of the main theorems, and
Appendix~\ref{sect:litguide} describes differences in terminology,
notation, and conventions among the various papers.  

\subsection*{Acknowledgements}

This paper is based on a series of lectures given by the author at
RIMS (Kyoto) in October 2008.  I am deeply grateful to Syu Kato, Susumu Ariki, and Kyo Nishiyama for welcoming me to Kyoto and for their warm and generous hospitality during my stay there.  The work described here was carried out with support from NSF grant DMS-0500873.

\section{Prologue: Staggered Representations}
\label{sect:stagrep}

In this section, we carry out a construction that is parallel to, and
exhibits the main features of, that of staggered sheaves.  Sheaves do not
explicitly appear in this section, but the ideas in this section will be
revisited later.

Let $H$ be an algebraic group over an algebraically closed field $\Bbbk$. 
Assume for simplicity that either $\charac \Bbbk = 0$, or else that $H$ is
solvable.  In either case, $H$ admits a Levi decomposition $H = LU$, where
$L$ is reductive, and $U$ is the unipotent radical of $H$.  Moreover, the
rational representations of $L$ are all semisimple.  (If $H$ is solvable,
$L$ is a torus.)  Let $\Rep H$ denote the category of finite-dimensional
rational representations of $H$, and let $\Irr(H)$ denote the set of
(isomorphism classes of) irreducible rational representations. 

We now describe a recipe for attaching an integer, called the \emph{step},
to each $V \in \Irr(H)$.   Choose, once and for all, a cocharacter 
\[
 \xi: \Gm \to \Cento(L),
\]
where $\Cento(L)$ denotes the identity component of the center of $L$, and
assume that
\begin{eqenum}
\item $\la\xi, \lambda\ra < 0$ for all weights $\lambda$ of $\Cento(L)$ on
$\Lie(U)$. \label{it:split-rep}
\end{eqenum}
Next, recall that $U$ acts trivially on any irreducible
representation of $H$, so $\Irr(H)$ can be identified with $\Irr(L)$.  In any
irreducible representation $V$ of $L$, the subgroup $\Cento(L)$ acts by
a character $\chi_V: \Cento(L) \to \Gm$.  We define the step of $V$ by
\[
\step V = \la \xi, \chi_V \ra.
\]

Next, let $\cD = D^b(\Rep H)$ denote the bounded derived category of $\Rep
H$.   We regard $\Rep H$ as a subcategory of $\cD$ as usual, by thinking of
objects of $\Rep H$ as chain complexes concentrated in degree $0$. 
Objects of the form $V[n]$, where $V \in \Irr(H)$, may
be plotted on a two-dimensional grid whose horizontal axis indicates step,
and whose vertical axis indicates cohomological degree in $\cD$:
\[
\incgr{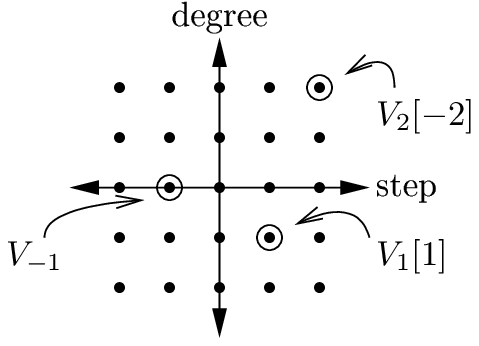}
\]
Of course, not all objects in $\cD$ live on points on this grid, but $\cD$
is generated by such objects: all chain complexes are built up by
extensions from complexes concentrated in a single degree, and all
$H$-representations are built up from irreducible representations.  Indeed,
the subcategory $\Rep H$ may be pictured thus:
\[
\incgr{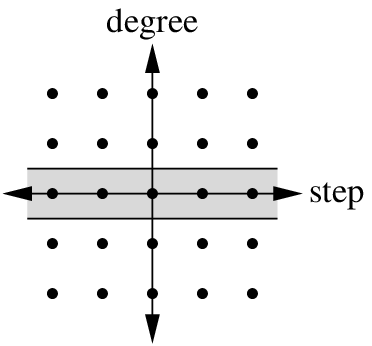}
\]
Let $\cM(H)$ denote the subcategory generated by objects of the form
\[
V[\step V],
\qquad \text{where $V \in \Irr(H)$.}
\]
In our picture of $\cD$, $\cM(H)$ is generated by objects lying along a line of slope $-1$:
\[
\incgr{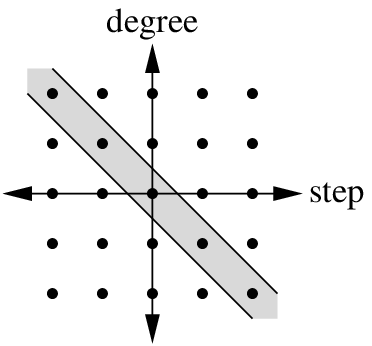}
\]

\begin{thm}
$\cM(H)$ is a semisimple abelian subcategory of $\cD$.  The simple objects are precisely the objects of the form $V[\step V]$, with $V \in \Irr(H)$.
\end{thm}
\begin{proof}[Proof sketch]
The usual way to produce abelian subcategories of triangulated categories is to use the machinery of \emph{$t$-structures}~\cite{bbd}.  A $t$-structure on $\cD$ is pair of subcategories $(\cD^{\le 0}, \cD^{\ge 0})$, satisfying a short list of axioms that will be recalled below.  

In our case, we define $\cD^{\le 0}$ (resp. $\cD^{\ge 0}$) to be the subcategory generated by objects $V[n]$ with $V \in \Irr(H)$ and $n \ge \step V$ (resp.~$n \le \step V$).  Thus:
\[
\cD^{\le 0}: \incgr{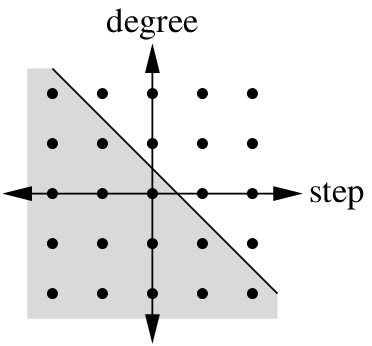}
\qquad 
\cD^{\ge 0}: \incgr{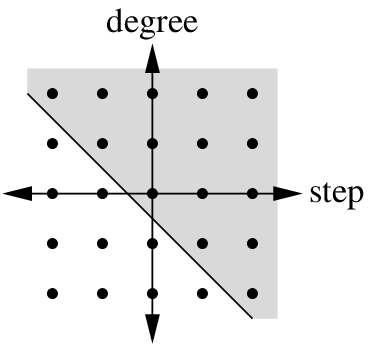}
\]
The three axioms to be checked are as follows:

(1)~\emph{$\cD^{\le 0}[1] \subset \cD^{\le 0}$ and $\cD^{\ge 0} \subset
\cD^{\ge 0}[1]$}.  This is obvious from the definition. 

(2)~\emph{$\Hom(A,B[-1]) = 0$ for all $A \in \cD^{\le 0}$ and $B \in
\cD^{\ge 0}$}.  By a standard induction argument, this can be reduced to
showing for any $V, W \in \Irr(H)$, we have $\Hom(V[n], W[m]) = 0$ whenever
$n \ge \step V$ and $m < \step W$.  Equivalently, we must show that
$\Ext^{m-n}(V,W)$ vanishes.  If $m - n < 0$, this is trivial, but if $m - n
\ge 0$, this must be proved using the assumption~\eqref{it:split-rep}. 

(3)~\emph{For any object $X \in \cD$, there is a distinguished triangle $A
\to X \to B \to$ with $A \in \cD^{\le 0}$ and $B[1] \in \cD^{\ge 0}$}.  The
proof of this proceeds by induction on the number of nonzero cohomology
objects of $X$ and a diagram-chasing argument using the
$9$-lemma~\cite[Proposition~1.1.11]{bbd}. 

The proof of the semisimplicity of $\cM(H)$ and the determination of simple
objects both come down to $\Hom$-group calculations similar to those in
part~(2) above. 
\end{proof}

\section{Review of Perverse Sheaves and Weights}
\label{sect:perv}

In this section, we briefly review the most important facts about perverse
sheaves, at least from the viewpoint of applications in representation
theory.  Subsequent sections do not depend on any facts listed here;
rather, the list is meant to serve as motivation for later results on
staggered sheaves.

Let $X$ be a variety over the algebraic closure
$\overline \F_q$ of a finite field $\F_q$.  Suppose $X$ is endowed with
a fixed stratification $\cS$ into finitely many locally closed smooth
strata, as well as a finite collection $\cL$ of isomorphism classes of
local systems on those strata.  We assume that the pair $(S,\cL)$
satisfies the conditions of~\cite[\S 2.2.10(a)--(c)]{bbd}, which are
technical conditions meant to ensure that the usual sheaf functors behave
well.

Fix a prime number $\ell$ different from the characteristic of $\F_q$, and
let $\Dbm X$ denote the ``derived'' category of bounded mixed constructible
complexes of \'etale $\Qlb$-sheaves.  (In the presence of a group action,
one should instead take $\Dbm$ to denote the Bernstein--Lunts equivariant
derived category~\cite{bl:esf}.)  For each stratum $S \in
\cS$, let $i_S: S \hto X$ denote the inclusion map.  Next, we define two
new full subcategories of $\Dbm X$ as follows:
\begin{align*}
\p\Dbml X0 &= \{ F \mid \text{$h^k(i_S^*F) = 0$ if $k > -\dim S$ for all $S
\in \cS$} \}, \\
\p\Dbmg X0 &= \{ F \mid \text{$h^k(i_S^!F) = 0$ if $k < -\dim S$ for all $S
\in \cS$} \}.
\end{align*}
The category of \emph{perverse sheaves}, denoted $\sP(X)$, is defined as follows:
\[
\sP(X) = \p\Dbml X0 \cap \p\Dbmg X0.
\]
We begin with a few basic facts about $\sP(X)$.

\begin{eqenum}
\item The pair $(\Dbml X0, \Dbmg X0)$ is a nondegenerate, bounded $t$-structure on $\Dbm X$, and therefore $\sP(X)$ is an abelian category. \label{it:tstruc}
\item Every perverse sheaf has finite length, and there is a bijection
\[
\left\{\begin{array}{c}
\text{simple objects}\\
\text{in $\sP(X)$}
\end{array}\right\}
\overset{\sim}\longleftrightarrow
\left\{
(C,L) \,\Big|
\begin{array}{c}
\text{$C \in \orb X$, and $L$ an irreducible}\\
\text{local system on $C$}
\end{array}\right\}
\]
The simple object corresponding to a pair $(C,L)$, denoted $\IC(C,L)$, is supported on $\overline C$, and its restriction to $C$ is $L[\dim C]$. \label{it:ic}
\item {\bf Poincar\'e--Verdier Duality}: The Verdier duality functor $\vD: \Dbm X \to \Dbm X$ preserves $\sP(X)$, and
\[
\vD \IC(C,L) \simeq \IC(C, L^*),
\]
where $L^*$ is the dual local system to $L$. \label{it:dual}
\setcounter{pervenum}{\value{enumi}}
\end{eqenum}

Next, we turn to the theory of weights.  Recall that being \emph{mixed}
means, roughly, that the Frobenius morphism acts on stalks at $\F_q$-points
of $X$ with well-behaved eigenvalues.  Specifically: all eigenvalues of
Frobenius are algebraic integers whose complex absolute values (under any
embedding $\overline\Q \hto C$) are of the form $q^{w/2}$ with $w \in \Z$. 
When $F$ is actually a sheaf (and not a complex), the integers $w$ that
occur are called the \emph{weights} of $F$.

The definition of weights for general objects of $\Dbm X$ is slightly
different.  The full subcategory of objects of weight ${}\le w$, denoted
$\Dwl w$, is defined by 
\[
\Dwl w = \{ F \in \Dbm X \mid \text{for each $k$, $h^k(F)$ has weights ${}\le w+k$} \}.
\]
There is also a category $\Dwg w$ of objects of weight ${}\ge w$.  When
$X$ is smooth, it can be defined by simply reversing the inequalities
above, but in general, that is not the correct definition.  Instead, it is
defined by
\begin{equation}\label{eqn:wtdual}
\Dwg w = \vD(\Dwl {-w}).
\end{equation}

By plotting sheaf weights along the horizontal axis, we may draw a picture
of $\Dwl w$ as shown below, in the spirit of the pictures in
Section~\ref{sect:stagrep}.  If $X$ is smooth, we may do likewise for
$\Dwg w$.
\[
\Dwl w: \incgr{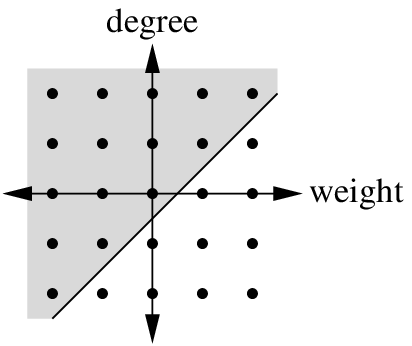}
\qquad
\txt{\strut \\ $\Dwg w$:\\ ($X$ smooth)} \incgr{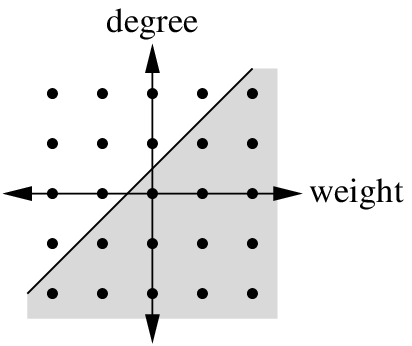}
\]
However, the latter picture can be misleading for nonsmooth $X$.

Finally, an object $F \in \Dbm X$ is \emph{pure} of weight $w$ if
\[
F \in \Dwl w \cap \Dwg w.
\]
Some important properties of the weight categories $\Dwl w$ and $\Dwg w$ are as follows.  If $X$ is smooth, the constant sheaf $\Qlb$ on $X$ is pure of weight $0$, but that does not hold for general $X$.

\begin{eqenum}
\item $\Dwl w[1] = \Dwl {w+1}$ and $\Dwg w[1] = \Dwg {w+1}$. \label{it:wtshift}
\item If $F \in \Dwl w$ and $G \in \Dwg {w+1}$, then $\Hom^i(F,G) = 0$ for all $i > 0$. \label{it:wtext}
\item (Deligne~\cite{del:lcw}) For a proper morphism $f: X \to Y$, the
functor $f_! = f_*$ takes pure objects to pure objects.
\label{it:wtweil}
\end{eqenum}

The last item is part of Deligne's proof of the Weil conjectures. 
Specifically, when $X$ is a smooth projective variety, it follows
from~\eqref{it:wtweil} (by taking $Y$ to be a point) that each $\ell$-adic
cohomology group $H^k(X,\Qlb)$ is pure of weight $k$. 

\begin{eqenum}
\item If $F$ and $G$ are perverse sheaves with $F \in \Dwl w$ and $G \in \Dwg {w+1}$, then in fact $\Hom^i(F,G) = 0$ for all $i \ge 0$.
\item \df{Purity:} Every perverse sheaf $F$ admits a canonical finite
filtration
\[
\cdots \subset F_{-1} \subset F_0 \subset F_1 \subset \cdots
\]
such that $F_w/F_{w-1}$ is pure of weight $w$.  In particular, every simple perverse sheaf is pure.
\item \df{Decomposition:} Any pure object in $\Dbm X$ is a direct
sum of shifts of simple perverse sheaves.
\end{eqenum}
For applications in representation theory, this Decomposition Theorem is usually used in conjunction with~\eqref{it:wtweil}.  

Finally, we consider ``standard'' and ``costandard'' objects.  Precise definitions will be given in Section~\ref{sect:quasi}; for now, one may keep in mind that Verma modules are standard objects in category $\cO$ for a complex semisimple Lie algebra.  Indeed, Verma modules are perhaps the motivating example for the general notion. 

\begin{eqenum}
\item Assume that each stratum $S$ is isomorphic to an affine space
$\bA^k$.  Then every simple perverse sheaf has a standard cover and a
costandard hull.  Moreover, $\sP(X)$ has enough projectives and enough
injectives. \label{it:aff-qh}
\end{eqenum}

\section{Preliminaries}
\label{sect:prelim}

\subsection{Notation and assumptions}
\label{subsect:not}

Let $\Bbbk$ be an algebraically closed field.  We temporarily assume that $\charac \Bbbk = 0$ (this will soon be dropped).  Let us put:

\medskip

\begin{tabular}{@{}p{2cm}l}
$G$ & a reductive algebraic group over $\Bbbk$ \\
$X$ & a variety over $\Bbbk$ on which $G$ acts with finitely many orbits \\
$\orb X$ & the set of $G$-orbits in $X$ \\
$\cg X$ & the category of $G$-equivariant coherent sheaves on $X$ \\
$\Db X$ & the bounded derived
category of $\cg X$ \\
$\omega_X$ & the equivariant dualizing complex of $X$ in $\Db X$
\end{tabular}

\medskip

\noindent
To be more precise, we normalize $\omega_X$ in the following way:
If $X$ is smooth, then
\[
\omega_X \simeq \cano X[\dim X],
\]
where $\cano X$ denotes the canonical bundle of $X$.  If $X$ is not smooth, the Sumihiro embedding theorem~\cite{sum:ec} tells us that we may find a closed $G$-equivariant embedding $\iota: X \hto Y$ into a smooth $G$-variety $Y$.  We then set
\[
\omega_X = R\iota^!\omega_Y.
\]
This object is independent of the choice of $\iota$.

We will also require the bounded-above and bounded-below derived
categories $\Dm X$ and $\Dp X$.  Recall~\cite{har:rd} that the latter must
be defined to consist of bounded-below chain complexes of
\emph{quasicoherent} sheaves with coherent cohomology.

For each orbit $C \in \orb X$, we define the following:

\medskip

\begin{tabular}{@{}p{2cm}l}
$i_C: C \hto X$ & the inclusion of $C$ as a
reduced locally closed subscheme \\
$H_C$ & the $G$-stabilizer of some point of $C$ \\
$H_C = L_CU_C$ & a Levi decomposition of $H_C$ \\
$\Cento(L_C)$ & the identity component of the center of the Levi factor of
$H_C$ \\
$\fu_C$ & the Lie algebra of the unipotent radical of $H_C$ \\
$\cI_C$ & the ideal sheaf in $\cO_X$ for the closed subvariety $\overline
C$ \\
$N^*(C)$ & conormal bundle of $C \subset X$
\end{tabular}

\medskip

\noindent
Above, the assumption that $\charac \Bbbk = 0$ was used in two places: the invocation of the Sumihiro embedding theorem for the definition of $\omega_X$, and the existence of a Levi decomposition for $H_C$.  We would now like to drop the assumption that $\charac \Bbbk = 0$, so if $\charac \Bbbk \ne 0$, we explicitly impose the following additional assumptions
\begin{eqenum}
\item $X$ admits a closed embedding $\iota: X \hto Y$ into a
smooth $G$-variety $Y$. \label{it:embed}
\item Each $H_C$ is solvable. \label{it:solv}
\end{eqenum}
Condition~\eqref{it:embed} is really not essential (see
Section~\ref{subsect:duality}), but in the absence of~\eqref{it:solv}, the
Decomposition Theorem (Theorem~\ref{thm:decomp}) fails.
Both hold in the important example where $G$ is a Borel subgroup
of reductive group and $X$ is a flag variety.

With these assumptions, we now have in any characteristic:
\begin{eqenum}
\item All representations of the reductive groups $L_C$ are completely
reducible. \label{it:semis}
\end{eqenum}

Recall that on a single $G$-orbit $C$, every equivariant coherent sheaf
is necessarily locally free.  In other words, $\cg
C$ is in fact the category of $G$-equivariant vector bundles on $C$, which
in turn is equivalent to the category $\Rep{H_C}$ of rational
representations of the isotropy group $H_C$.  We will freely make use of
the equivalence
\[
 \cg C \simeq \Rep{H_C}
\]
in the sequel.  In particular, the Lie algebra $\fu_C$ equipped with the
adjoint action of $H_C$ may be regarded as a coherent sheaf on $C$, and
the vector bundle $N^*(C)$ may be regarded as an $H_C$ representation.

\subsection{Definition of staggered sheaves}
\label{subsect:stagdefn}

The category of staggered sheaves depends on the following choices:
\begin{eqenum}
 \item For each orbit $C \in \orb X$, a cocharacter
$\xi_C: \Gm \to \Cento(L_C)$ such that
\[
 \la \xi_C, \lambda \ra \le -1
\qquad
\text{for all weights $\lambda$ of $\Cento(L_C)$ on $\fu_C$ and on
$N^*(C)$.}
\]
Such a collection $\{\xi_C \mid C \in \orb X \}$ is called
an \emph{$s$-structure} on $X$. \label{it:recsplit}
\item A function $r: \orb X \to \Z$.  This will be known as the
\emph{perversity function}.
\end{eqenum}
Fix an $s$-structure and a perversity once and for all.  

\begin{defn}
Let $\cF \in \cg C$ be an irreducible vector bundle, and let $\chi_\cF:
\Cento(L_C) \to \Gm$ be the character of $\Cento(L_C)$ on the corresponding
$H_C$-representation.  The \emph{step} of $\cF$ is defined by
\[
\step \cF = \la \xi_C, \chi_\cF\ra.
\]
\end{defn}

The notion of \emph{step} does not make sense for general objects of $\cg
C$.  Nevertheless, we can form filtrations of $\cg C$ in terms of steps of
irreducible objects.  For $w \in \Z$, we define full subcategories of $\cg
C$ as follows:
\begin{align*}
\cgl Cw &= \{ \cF \mid \text{for every irreducible subquotient $\cG$ of
$\cF$, $\step \cG \le w$} \}, \\
\cgg Cw &= \{ \cF \mid \text{for every irreducible subquotient $\cG$ of
$\cF$, $\step \cG \ge w$} \}.
\end{align*}
Using these, we define two full subcategories of $\Db X$ as follows:
\begin{align*}
\ru\Dl X0 &= \{ \cF \mid \text{$h^k(Li_C^*\cF) \in \cgl C{r(C)-k}$ for all
$C \in \orb X$ and all $k \in \Z$} \}, \\
\ru\Dg X0 &= \{ \cF \mid \text{$h^k(Ri_C^!\cF) \in \cgg C{r(C)-k}$ for all
$C \in \orb X$ and all $k \in \Z$} \}.
\end{align*}
(The unbounded versions $\ru\Dml X0$ and $\ru\Dpg X0$ are defined similarly.)

\begin{defn}
The category of \emph{staggered sheaves} on $X$, denoted $\ru\cM(X)$, or simply $\cM(X)$, is the category $\ru\Dl X0 \cap \ru\Dg X0$.
\end{defn}

It is clear than when $X$ consists of a single orbit $C$, the category
$\cM(C)$ is equivalent (up to shift) to the category $\cM(H_C)$ of
``staggered representations'' considered in Section~\ref{sect:stagrep}, and
we may draw the following pictures:
\[
\begin{array}{c@{\quad}c@{\quad}c}
\incgr{stagtneg} & \incgr{stagtpos} &\incgr{stagt} \\
\ru\Dl C0\qquad & \ru\Dg C0\qquad & \cM(C)^{\strut}\qquad
\end{array}
\]

When $X$ contains more than one orbit, the structure of these categories is
not so clear.  This topic will be explored further in
Section~\ref{subsect:barstag}.

\section{Basic Properties}
\label{sect:basic}

Perhaps the biggest technical difficulty in getting the theory of staggered
sheaves off the ground is that the restriction functor $j^*$, where $j: U
\hto X$ is an open inclusion, does not have adjoints in the setting of
coherent sheaves.  (The direct image functor $j_*$ gives quasicoherent
sheaves in general, where the extension-by-zero functor $j_!$ takes values
in the category of inverse limits of coherent sheaves.)

\begin{thm}\label{thm:tstruc}
$(\ru\Dl X0, \ru\Dg X0)$ is a nondegenerate, bounded $t$-structure on $\Db X$, so $\cM(X)$ is an abelian category.
\end{thm}
\begin{proofcom}
This theorem cannot be proved using the method of ``gluing of
$t$-structures'' \cite[\S 1.4]{bbd} that is typically used
for perverse sheaves, because of the lack of adjoints mentioned above. 
Instead, we use the machinery of \emph{baric structures}, discussed in
Appendix~\ref{sect:baric}.
\end{proofcom}

\begin{thm}\label{thm:simple}
Every staggered sheaf has finite length, and there is a bijection
\[
\left\{\begin{array}{c}
\text{simple objects}\\
\text{in $\cM(X)$}
\end{array}\right\}
\overset{\sim}\longleftrightarrow
\left\{
(C,\cL) \,\Big|
\begin{array}{c}
\text{$C \in \orb X$, and $\cL$ an irreducible}\\
\text{vector bundle on $C$}
\end{array}\right\}
\]
The simple object corresponding to a pair $(C,\cL)$, denoted
$\ru\cIC(C,\cL)$, or simply $\cIC(C,\cL)$, is supported on $\overline C$,
and its restriction to $C$ is $\cL[\step \cL - r(C)]$. 
\end{thm}
\begin{proofcom}
Among all objects supported on $\overline C$ with the correct restriction
to $C$, the objects $\cF$ in the image of the functor $\cIC(C,\cdot)$ are
characterized by the property that for any $C' \subset \overline C \ssm C$,
we have 
\[
Li_{C'}^*\cF \in \ru\Dml {C'}{-1}
\qquad\text{and}\qquad
Ri_{C'}^!\cF \in \ru\Dpg {C'}{1}.
\]
This criterion can be used in examples to verify that some object is $\cIC(C,\cL)$
\end{proofcom}

Finally, we turn to duality.  In the coherent setting, the appropriate analogue of the Poincar\'e--Verdier duality functor is the Serre--Grothendieck duality functor $\sgD: \Db X \to \Db X$, developed in~\cite{har:rd} and generalized to the equivariant case in~\cite{bez:pcs}.  To describe the behavior of $\cM(X)$ under $\sgD$, we require some auxiliary definitions.
Let us define a new function
\[
\bar r: \orb X \to \Z,
\]
known as the \emph{dual perversity} to the given perversity $r: \orb X \to \Z$, by the formula
\begin{equation}\label{eqn:dualr}
\bar r(C) = \step \cano C - \dim C - r(C).
\end{equation}
We may then define categories
\[
\bru\Dl X0, \quad \bru\Dg X0, \quad \bru\cM(X)
\]
in the same way was $\ru\Dl X0$, $\ru\Dg X0$, and $\ru\cM(X)$, but using $\bar r$ in place of $r$ (and using the same $s$-structure as before).

\begin{thm}\label{thm:duality}
$\sgD(\ru\Dl X0) = \bru\Dg X0$ and $\sgD(\ru\Dg X0) = \bru\Dl X0$.  In particular, $\sgD(\ru\cM(X)) = \bru\cM(X)$, and
\[
\sgD(\ru\cIC(C,\cL)) \simeq \bru\cIC(C,\cL^\vee),
\]
where $\cL^\vee = \cHom(\cL,\cano C)$ is the Serre dual vector bundle to $\cL$.
\end{thm}

This statement becomes much cleaner, and closer to~\eqref{it:dual}, when it happens that $r = \bar r$.  Clearly, such a self-dual perversity exists if and only if the integer
\[
\step \cano C - \dim C,
\]
known as the \emph{staggered codimension} of $C$, is even for all $C \in \orb X$.  It is not known under what conditions this holds.  (Similar considerations pertain to perverse sheaves as well, although in the algebraic setting, there is always a self-dual perversity---we silently chose it in Section~\ref{sect:perv}---because the $\ell$-adic
cohomological dimension of every variety is even.)

\section{Purity and Decomposition}
\label{sect:puredec}

We now describe a family of subcategories of $\Db X$ that are analogous to
the weight categories $\Dwl w$ and $\Dwg w$ in $\ell$-adic sheaf theory. 
Given an integer $w$, define
\[
\ru\Dmkl Xw =
\left\{ \cF \Big| 
\begin{array}{c}
\text{$h^k(Li_C^*\cF) \in \cgl C{w+r(C)+\dim C+k}$} \\
\text{for all $C \in \orb X$ and all $k \in \Z$}
\end{array}\right\}.
\]
As usual, when there is no risk of ambiguity, this category will simply be
denoted $\Dmkl Xw$.  Its definition clearly bears a great deal of
resemblance to
that of $\ru\Dml Xw$; the most essential difference is that the step
constraint involves a ``$+k$'' rather than a ``$-k$.''  Next, we define
\[
\ru\Dpkg Xw = 
\left\{ \cF \Big| 
\begin{array}{c}
\text{$h^k(Ri_C^!\cF) \in \cgg C{w+r(C)+\dim C+k}$} \\
\text{for all $C \in \orb X$ and all $k \in \Z$}
\end{array}\right\}.
\]
We also have bounded versions $\Dkl Xw$ and $\Dkg Xw$.  Over a single
orbit, we may draw pictures of these categories resembling those for the
$\ell$-adic weight categories:
\[
\Dkl Xw: \incgr{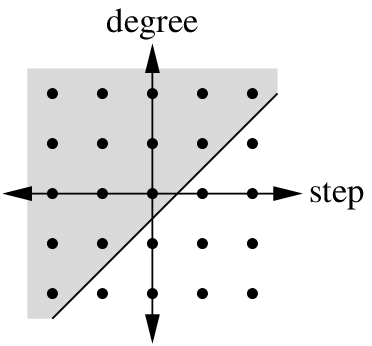}
\qquad
\Dkg Xw: \incgr{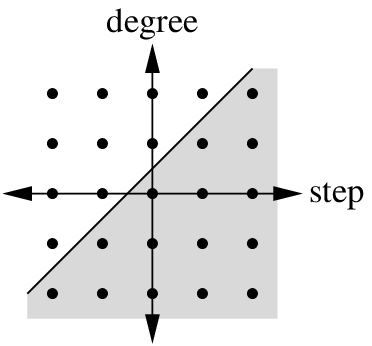}
\]
Of course, these categories are more complicated when $X$ contains more
than one orbit.  Like the $\ell$-adic weight categories, these categories
enjoy a duality property, cf.~\eqref{eqn:wtdual}
and~Theorem~\ref{thm:duality}:

\begin{thm}\label{thm:coduality}
$\sgD(\ru\Dkl Xw) = \bru\Dkg X{-w}$ and $\sgD(\ru\Dkg Xw) = \bru\Dkl
X{-w}$.
\end{thm}

Our first result about them is similar to, but stronger
than, the results~\eqref{it:wtshift} and~\eqref{it:wtext} for $\ell$-adic
sheaves.

\begin{thm}
The collection of subcategories $(\{\Dkl Xw\}, \{\Dkg Xw\})$ forms a co-$t$-structure on $\Db X$.  That is:
\begin{enumerate}
\item $\Dkl Xw \subset \Dkl X{w+1}$ and $\Dkg Xw \supset \Dkg X{w+1}$.
\item $\Dkl Xw[1] = \Dkl X{w+1}$ and $\Dkg Xw[1] = \Dkg X{w+1}$.
\item If $\cF \in \Dkl Xw$ and $\cG \in \Dkg X{w+1}$, then $\Hom(\cF,\cG) = 0$.
\item For any $\cF \in \Db X$, there is a distinguished triangle $\cF' \to \cF \to \cF'' \to$ with $\cF' \in \Dkl Xw$ and $\cF'' \in \Dkg X{w+1}$.
\end{enumerate}
\end{thm}

Note that the inclusions $\Dkl Xw \subset \Dkl Xw[1]$ and $\Dkg Xw
\supset \Dkg Xw[1]$ are the opposite of what one would have in a
$t$-structure. Also in contrast with $t$-structures, the distinguished
triangle in part~(4) above is not functorial in general.

\begin{defn}
An object $\cF \in \Db X$ is \emph{skew-pure} of \emph{skew-degree $w$} if it belongs to $\Dkl Xw \cap \Dkg Xw$.
\end{defn}

\begin{thm}[Purity]\label{thm:purity}
Every staggered sheaf $\cF$ admits a canonical finite filtration
\[
\cdots \subset \cF_{w-1} \subset \cF_w \subset \cF_{w+1} \subset \cdots
\]
such that $\cF_w/\cF_{w+1}$ is skew-pure of skew degree $w$.  In particular, every simple staggered sheaf is skew-pure.
\end{thm}
\begin{proofcom}
The hard part of this is showing that a simple object is skew-pure, and the
difficulty is that there is no general method to compute the restriction of
$\cIC(C,\cL)$ to $\overline C \ssm C$.  See Section~\ref{subsect:barpur}.
\end{proofcom}

The skew degree of a simple staggered sheaf $\cF$ is given by
\[
\skdeg \cIC(C,\cL) = 2(\step \cL - r(C)) - \dim C.
\]
The next result follows immediately from properties of co-$t$-structures.

\begin{prop}\label{prop:hom}
Let $\cF$ and $\cG$ be simple staggered sheaves, with $\skdeg \cF = v$ and $\skdeg \cG = w$.  Then $\Hom^k(\cF,\cG) = 0$ for $k > v - w$.
\end{prop}

Here, $\Hom^k(\cF,\cG)$ is a synonym for $\Hom(\cF, \cG[k])$.  It
can be identified with $\Ext^k(\cF,\cG)$ when $k \le 1$ (but not in general
when $k > 1$).  Thus, the proposition above includes an $\Ext^1$-vanishing
condition.  A stronger statement about $\Ext^1$-groups will appear in
Corollary~\ref{cor:ext1}.  See also Section~\ref{subsect:high-ext}.

\begin{thm}[Decomposition]\label{thm:decomp}
Every skew-pure object in $\Db X$ is a direct sum of shifts of simple staggered sheaves.
\end{thm}

\section{Quasi-hereditary Abelian Categories}
\label{sect:quasi}

Let $\cA$ be a finite-length $\Bbbk$-linear abelian category, and let $S$ be the set of isomorphism classes of simple objects.  Assume $S$ is endowed with a fixed total order $\le$.  For each $s \in S$, fix a representative object $L_s$.  Assume also that $\End(L_s) \simeq \Bbbk$ for each $s \in S$.

\begin{defn}
An object $M_s$ together with a surjective morphism $\phi_s: M_s \to L_s$ is called a \emph{standard cover} if
\begin{enumerate}
\item Every simple subquotient of $\ker \phi_s$ is isomorphic to some $L_t$ with $t < s$.
\item $\Hom(M_s, L_t) = \Ext^1(M_s,L_t) = 0$ for all $t < s$.
\end{enumerate}
Dually, an object $N_s$ together with an injective morphism $\psi_s: L_s \to N_s$ is called a \emph{costandard hull} if
\begin{enumerate}
\item Every simple subquotient of $\cok \psi_s$ is isomorphic to some $L_t$ with $t < s$.
\item $\Hom(L_t, N_s) = \Ext^1(L_t, N_s) = 0$ for all $t < s$.
\end{enumerate}
\end{defn}

Standard arguments show that standard covers and costandard hulls, when they exist, are unique up to canonical isomorphism.  As noted at the end of Section~\ref{sect:perv}, the motivating example of a standard object is a Verma modules in category $\cO$ for a complex semisimple Lie algebra.

Note that a standard cover of $L_s$ is a projective cover within the smaller category $\cA_{\le s}$ generated by objects $\{ L_t \mid t \le s\}$.  Similarly, a costandard hull is an injective hull within $\cA_{\le s}$.

\begin{defn}
The category $\cA$ is said to be \emph{quasi-hereditary} if
\begin{eqenum}
\item Every simple object admits a standard cover and a costandard hull.
\end{eqenum}
\end{defn}

This definition, taken from~\cite{bez:pcs}, is not the most common one. 
Many authors impose additional conditions, cf. Section~\ref{subsect:recip}.

\begin{thm}
\begin{enumerate}
\item $\cM(X)$ is quasi-hereditary.
\item $\cM(X)$ has enough projectives and enough injectives.
\end{enumerate}
\end{thm}
\begin{proofcom}
The proof of~\eqref{it:aff-qh} for perverse sheaves makes use of the
functors $j_!$ and $j_*$, where $j: U \hto X$ is an open inclusion.  These
are unavailable in the coherent setting.  However, it turns out that the
functor of abelian categories $j^*: \cM(X) \to \cM(U)$ has adjoints on
both sides.  These adjoints can be used in place of $j_!$ and $j_*$.
\end{proofcom}

For a simple staggered sheaf $\cIC(C,\cL)$, let $M(C,\cL)$ and $N(C,\cL)$ denote its standard cover and costandard hull, respectively.

\begin{thm}\label{thm:costd-struc}
Let $w = \skdeg \cIC(C,\cL)$.
\begin{enumerate}
\item The kernel of $M(C,\cL) \to \cIC(C,\cL)$ is skew-pure of skew degree $w - 1$.
\item The cokernel of $\cIC(C,\cL) \to N(C,\cL)$ is skew-pure of skew degree $w + 1$.
\end{enumerate}
\end{thm}

Note that by the Decomposition Theorem, the kernel and cokernel mentioned in this theorem are both necessarily semisimple.

We conclude with two results whose statements do not involve standard or costandard objects, but whose proofs do.

\begin{cor}\label{cor:ext1}
If $\Ext^1(\cIC(C,\cL),\cIC(C',\cL')) \ne 0$, then
\[
\skdeg \cIC(C',\cL') = \skdeg \cIC(C,\cL) - 1.
\]
\end{cor}

\begin{cor}\label{cor:semis}
For any orbit $C \in \orb X$, the functor $\cIC(C,\cdot): \cM(C) \to \cM(X)$ embeds $\cM(C)$ as a semisimple Serre subcategory of $\cM(X)$.
\end{cor}

The semisimplicity of $\cM(C)$ itself was essentially stated in Section~\ref{sect:stagrep}.  The additional content of this corollary is that even in the larger category $\cM(X)$, there are no nontrivial extensions about simple objects supported on the same orbit closure.

\section{An Example}
\label{sect:exam}

Let $G$ denote the usual Borel subgroup of $PGL(2,\Bbbk)$:
\[
G = \left\{\begin{bmatrix} * & * \\ & * \end{bmatrix} \right\}
\Big/
\left\{\pm\begin{bmatrix} 1 & \\ & 1 \end{bmatrix} \right\}
\]
Let $X = \bP^1$, endowed with the obvious action of $G$.  (Note that $X$ is the flag variety for $PGL(2,\Bbbk)$.  $G$ has two orbits on $\bP^1$: a one-point orbit, denoted $Z$, and its complement, an open orbit isomorphic to $\bA^1$ and denoted $U$.  The isotropy groups are $H_Z = G$ and $H_U \simeq \Gm$.  Irreducible representations of both these groups are in bijection with $\Z$.  For $n \in \Z$, let $V_n \in \cg Z$ and $\cL_n \in \cg U$ denote the corresponding vector bundles on orbits.

Note that the torus $H_U \subset G$ acts linearly on $U$ with weight $1$. 
It can be deduced from this that $N^*(Z) \simeq \fu_Z \simeq V_1$.  It also
follows that for any two integers $n,m \in \Z$, there exists a line bundle
$\cF(n,m) \in \cg X$ such that $\cF(n,m)|_U \simeq \cL_n$ and
$i_Z^*\cF(n,m) \simeq V_m$.  The canonical bundle is $\cano X = \cF(-1,1)$,
and Serre--Grothendieck duality is given by
\[
\sgD(\cF(n,m)) \simeq \cF(-1-n,1-m)[1].
\]

We impose an $s$-structure on $X$ by requiring
$\step \cL_n = n$ and $\step V_n = -n$, and we give it the self-dual perversity $r: U \mapsto -1$, $Z \mapsto 0$.
Then, the simple staggered sheaves are given by
\[
\cIC(U,\cL_n) \simeq \cF(n,-n)[n+1]
\qquad\text{and}\qquad
\cIC(Z,V_n) \simeq V_n[-n].
\]
The only nontrivial extensions among simple objects are the following:
\begin{gather*}
0 \to \cIC(Z,V_{-n}) \to \cF(n,-n+1)[n+1] \to \cIC(U,\cL_n) \to 0 \\
0 \to \cIC(U,\cL_n) \to \cF(n,-n-1)[n+1] \to \cIC(Z,V_{-n-1}) \to 0
\end{gather*}
The middle term of each of those is both projective and injective.  In fact, we have
\begin{align*}
M(U,\cL_n) \simeq P(U,\cL_n) &\simeq \cF(n,-n+1)[n+1] \simeq I(Z,V_{-n}) \\
N(U,\cL_n) \simeq I(U,\cL_n) &\simeq \cF(n,-n-1)[n+1] \simeq P(Z,V_{-n-1})
\end{align*}
Since $Z$ is closed, the simple objects $\cIC(Z,V_n)$ are also standard and costandard.

The structure of the category $\cM(X)$ can be summarized in the following diagram, in which an arrow $\cF \to \cG$ means that $\Ext^1(\cG,\cF) \ne 0$.
\[
\xymatrix@=0pt{
\skdeg: & & -2 & -1 & 0 & 1 & 2 & 3 \\
U: &&&\cIC(U,-1)\ar[ddl] & & \cIC(U,0) \ar[ddl] & & \cIC(U,1)\ar[ddl] \\
&\cdots &&&&&&&\cdots \\
Z: && \cIC(Z,1) & & \cIC(Z,0) \ar[uul] & & \cIC(Z,-1) \ar[uul]
}
\]
It can also be checked by direct calculation that derived tensor products
\[
\cIC(Z,V_n) \Lotimes \cIC(Z,V_m)
\]
are skew-pure, of skew-degree $-2n-2m$.  The Decomposition Theorem tells us these must be semisimple; we have
\[
\cIC(Z,V_n) \Lotimes \cIC(Z,V_m)
\simeq
\cIC(Z, V_{n+m}) \oplus \cIC(Z, V_{n+m+1})[2].
\]
Further explicit examples appear in~\cite[\S 11]{a:sts}, \cite[\S
12]{at:pdtss}, and~\cite[\S 9]{a:qhpss}.  The example in~\cite[\S
12]{at:pdtss} gives another illustration of the Decomposition Theorem. 
The example in~\cite[\S 9]{a:qhpss} involves a nonsmooth variety.

\section{Further Questions}
\label{sect:ques}

\subsection{Reciprocity formulas}
\label{subsect:recip}

Let $\cA$ and $S$ be as in Section~\ref{sect:quasi}.
The notion of quasi-hereditarity often includes the following additional
assumptions:
\begin{eqenum}
\item The set $S$ of isomorphism classes of simple objects is finite.
\label{it:finsimp}
\item For any standard object $M_s$ and costandard object $N_t$,
$\Ext^2(M_s,N_t) = 0$. \label{it:hered}
\end{eqenum}
(Both of these fail for $\cM(X)$ in general.)  These extra conditions
always imply the existence of enough projectives and
injectives~\cite[Theorem~3.2.1]{bgs:kdprt} (see also~\cite{cps:fda}).  Moreover, every indecomposable projective admits a \emph{standard filtration}, i.e., a filtration whose subquotients are standard objects.
The multiplicities in a standard filtration obey the celebrated
\emph{Brauer--Humphreys reciprocity} formula:
\begin{equation}\label{eqn:recip}
(P_s : M_t) = [ N_t : L_s ].
\end{equation}
Perhaps the most famous instance of this is the one known as \emph{BGG
reciprocity} in category $\cO$ for a complex semisimple Lie algebra.

The known examples of staggered sheaves include one in which projectives do have standard filtrations (see~\cite[\S 9]{a:qhpss}) and others (including the one in Section~\ref{sect:exam}) in which they do not.  It would be nice to find a characterization of cases in which this happens, and therefore in which the reciprocity formula~\eqref{eqn:recip} holds.

\subsection{Higher Ext-groups}
\label{subsect:high-ext}

It can be seen in the smallest examples, including the one in Section~\ref{sect:exam}, that the bounded derived category of $\cM(X)$ need not be equivalent to $\Db X$, so care must be taken in passing from higher Hom-groups in $\Db X$ to higher Ext-groups for the abelian category $\cM(X)$.  If $\cF$ and $\cG$ are simple staggered sheaves with $\skdeg \cF \le \skdeg \cG$, Proposition~\ref{prop:hom} tells us that $\Hom^k(\cF,\cG) = 0$ for all $k > 0$, and then it follows from~\cite[Remarque~3.1.17]{bbd} that
\begin{eqenum}
\item $\Ext^k(\cF,\cG) = 0$
for all $k > 0$ if $\skdeg \cF \le \skdeg \cG$. \label{it:mixed}
\end{eqenum}
It is natural to ask what can be said about higher $\Ext$-groups when $\skdeg \cF > \skdeg \cG$.  It follows from~\eqref{it:mixed} that $\cM(X)$ is a \emph{mixed category} in the sense of~\cite[\S 4]{bgs:kdprt}.  In keeping with the themes of that paper, one may ask whether $\cM(X)$ is, in fact, a \emph{Koszul category}.  In other words, is it true that
\[
\Ext^k(\cF,\cG) = 0
\qquad
\text{if $\skdeg \cF - \skdeg \cG \ne k$?}
\]
Corollary~\ref{cor:ext1} says this holds for $k = 1$; it seems reasonable to conjecture that it holds in general.

\subsection{Cohomological dimension of $\cM(X)$}
\label{subsect:cohom-dim}

It is well-known that the cohomological dimension of the category $\cg X$
is $\dim X$ if $X$ is smooth and infinite otherwise.  For staggered
sheaves, the situation is quite different: in the example considered in
Section~\ref{sect:exam}, $\cM(X)$ has infinite cohomological dimension, but
in every known example on an affine variety, including a nonsmooth
one~\cite[\S 9]{a:qhpss}, $\cM(X)$ has finite cohomological dimension.  It
would be interesting to find necessary and sufficient conditions for
$\cM(X)$ to have finite cohomological dimension, and to compute its
dimension in those cases.

\appendix
\section{Baric Structures}
\label{sect:baric}

\subsection{Overview}
\label{subsect:barover}

A vital role in the proofs of all the main results on staggered sheaves is
played by the notion of a \emph{baric structure}, introduced by the
author and D.~Treumann in~\cite{at:bstc}.  The motivating example of a
baric structure appeared earlier, in work of S.~Morel, in the $\ell$-adic
setting~\cite{mor:cic}.  The definition is given as part of Theorem~\ref{thm:baric} below.

The paper~\cite{at:bstc} includes an extensive study of so-called \emph{HLR baric structures}.  These baric structures are \emph{hereditary} (well-behaved on closed subschemes), \emph{local} (well-behaved on open subschemes), and \emph{rigid} (well-behaved on nilpotent thickenings).  In additional to the general theory, that paper includes a construction of a specific family of HLR baric structures, which we now describe.

Let $q: \orb X \to \Z$ be a function, and define a collection of full
subcategories of $\Dm X$ and $\Dp X$, respectively, as follows:
\begin{align*}
\q\Dmsl Xw &= \{ \cF \mid \text{$h^k(Li_C^*\cF) \in \cgl C{q(C)+w}$ for all
$C \in \orb X$ and all $k \in \Z$} \}, \\
\q\Dpsg Xw &= \{ \cF \mid \text{$h^k(Ri_C^!\cF) \in \cgg C{q(C)+w}$ for all
$C \in \orb X$ and all $k \in \Z$} \}.
\end{align*}
As usual, let $\q\Dsl Xw$ and $\q\Dsg Xw$ denote the bounded versions of
these categories.  These categories are stable under shift.  Over a single orbit, they look like this:
\[
\q\Dsl Cw: \incgr{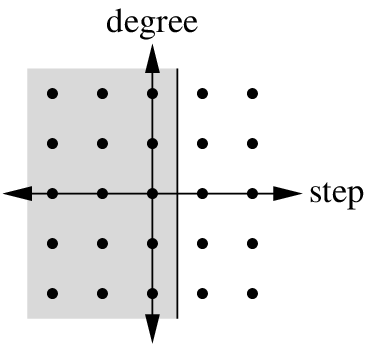}
\qquad
\q\Dsg Cw: \incgr{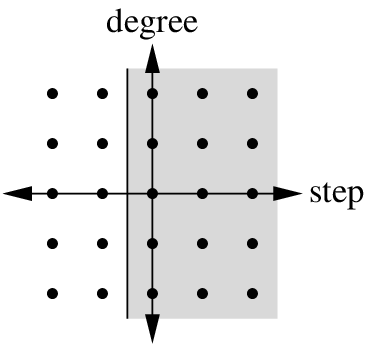}
\]
These categories exhibit a duality property similar to those in Sections~\ref{sect:basic} and~\ref{sect:puredec}.

\begin{thm}\label{thm:barduality}
$\sgD(\q\Dsl Xw) = \cheq\Dsg X{-w}$ and $\sgD(\q\Dsg Xw) = \cheq\Dsl X{-w}$, where $\hat q(C) = \step \cano C - q(C)$.
\end{thm}

\begin{thm}\label{thm:baric}
The collection of subcategories $(\{\q\Dsl Xw\},\{\q\Dsg Xw\})$ forms a
baric structure on $\Db X$.  That is:
\begin{enumerate}
\item $\q\Dsl Xw \subset \q\Dsl X{w+1}$ and $\q\Dsg Xw \supset \q\Dsg
X{w+1}$.
\item If $\cF \in \q\Dsl Xw$ and $\cG \in \q\Dsg X{w+1}$, then
$\Hom(\cF,\cG) = 0$.
\item For any $\cF \in \Db X$, there is a distinguished triangle $\cF' \to
\cF \to \cF'' \to$ with $\cF' \in \q\Dsl Xw$ and $\cF'' \in \q\Dsg X{w+1}$.
\end{enumerate}
\end{thm}

In a baric structure, as in a $t$-structure (but unlike in a
co-$t$-structure), the distinguished triangle in the last property above is
functorial.  Specifically, there are \emph{baric truncation functors}
\[
\q\beta_{\le w}: \Db X \to \q\Dsl Xw
\qquad\text{and}\qquad
\q\beta_{\ge w}: \Db X \to \q\Dsg Xw.
\]

\subsection{Staggering operation}
\label{subsect:barstag}

The baric structure above, like the $t$-structure of Section~\ref{subsect:stagdefn} and the co-$t$-structure of Section~\ref{sect:puredec}, is defined by ``orbitwise'' conditions, but all three of these admit an alternate description by ``cohomologywise'' conditions.  Define a subcategory of
$\cg X$ as follows:
\[
\q\cgl Xw = \{ \cF \mid \text{$i_C^*\cF \in \cgl C{q(C)+w}$ for all $C \in \orb X$} \}.
\]
The descriptions appearing in the proposition below were actually taken as definitions in~\cite{at:bstc, at:pdtss} (and the orbitwise descriptions required proof).

\begin{prop}\label{prop:cohom}
Let $\cF \in \Db X$.
\begin{enumerate}
\item $\cF \in \q\Dsl Xw$ if and only if $h^k(\cF) \in \q\cgl Xw$ for all
$k \in \Z$.
\item $\cF \in \ru\Dl Xw$ if and only if $h^k(\cF) \in {}_{r}\cgl
X{w-k}$ for all $k \in \Z$.
\item $\cF \in \ru\Dkl Xw$ if and only if $h^k(\cF) \in {}_{\llcorner r \lrcorner}\cgl X{w+k}$
for all $k \in \Z$, where
\[
 \llcorner r \lrcorner(C) = r(C) + \dim C.
\]
\end{enumerate}
\end{prop}

This really is a one-sided statement: in general, membership in $\q\Dsg Xw$, $\ru\Dg X0$, or $\ru\Dkg Xw$ cannot directly be tested for on cohomology sheaves.  For $\q\Dsl Xw$, $\ru \Dl X0$, and $\ru\Dkl Xw$, however, it is now reasonable to draw pictures even in the many-orbit case, with the horizontal axis now indicating membership in a $\q\cgl Xw$ rather than step.
\[
\begin{array}{@{}c@{}c@{}c@{}}
\incgr{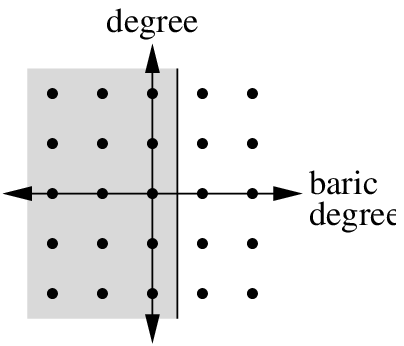} & \incgr{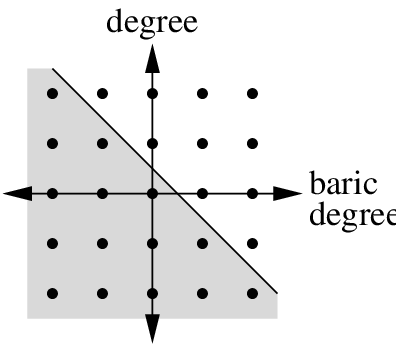} & \incgr{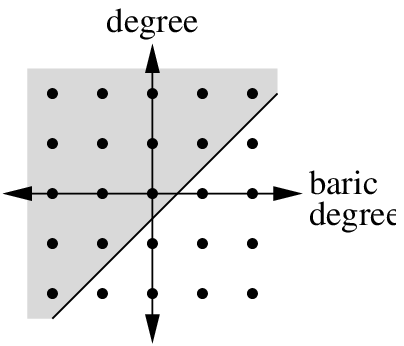} \\
\q\Dsl Xw\quad & \ru\Dl X0\quad & \ru\Dkl Xw\quad
\end{array}
\]
The advantage of the descriptions in Proposition~\ref{prop:cohom} is that the relationship between $\q\Dsl Xw$ and $\ru\Dl X0$ can now be described in language that makes sense in any triangulated category.

The \emph{staggering operation} is a procedure that takes ``compatible''
baric and $t$-structures on a triangulated category, and produces from
them a new $t$-structure.  The latter is defined by cohomology conditions.  Theorem~\ref{thm:tstruc} is proved by invoking this general procedure.

\subsection{Baric purity}
\label{subsect:barpur}

The full subcategory $\q\Dsp Xw = \q\Dsl Xw \cap \q\Dsg Xw$, shown below for a single orbit, is a triangulated category in its own right, and one may ask whether it admits any interesting $t$-structures.
\[
\q\Dsp Cw: \incgr{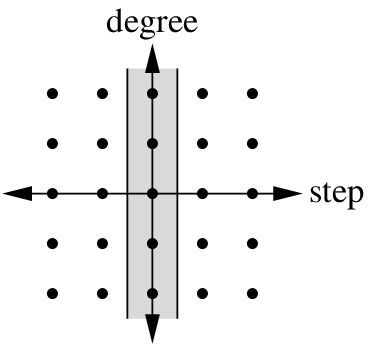}
\]
Two $t$-structures on $\q\Dsp Xw$ (known as the \emph{purified} and the \emph{pure-perverse} $t$-structures) are defined and studied in~\cite{at:pdtss}.  Neither of these is induced by one on $\Db X$, but they nevertheless turn out to be useful for studying particular objects.  Specifically, the key to proving the Purity Theorem, Theorem~\ref{thm:purity}, is showing that a simple staggered sheaf $\cIC(C,\cL)$ also lies in the heart of the pure-perverse $t$-structure on $\q\Dsp Xw$ for a suitable choice of $q$ and $w$.

Along the way, it is shown in~\cite{at:pdtss} that when the perversity functions obey certain inequalities, baric versions of the Purity and Decomposition Theorems hold.

\begin{thm}
Under certain assumptions on $r$ and $q$, the following hold.
\begin{enumerate}
\item Every staggered sheaf admits a canonical finite filtration with baric-pure subquotients.  In particular, every simple staggered sheaf is baric-pure.
\item Every baric-pure object in $\Db X$ is a direct sum of shifts of simple staggered sheaves.
\end{enumerate}
\end{thm}

\subsection{Nonreduced schemes}
\label{subsect:nonred}

Although $X$ was assumed to be a (reduced) variety throughout in the
present note, it is in fact essential to allow nonreduced schemes in general, since the scheme-theoretic support of a coherent sheaf on a variety may be nonreduced.  The ``rigidity'' property of HLR baric structures plays an important role here: for nonreduced $X$, membership in (for instance) $\q\Dsl Xw$ can be tested for after pulling back to $\q\Dsl {X_\red}w$, where $X_\red$ is the associated reduced scheme.  (This was not known when~\cite{a:sts} was written, so definitions in that paper are typically quantified over all closed subschemes, rather than just all reduced closed subschemes.)

However, the following result from~\cite{a:qhpss} depends on the quasi-hereditary property and does not seem to follow directly from rigidity.

\begin{thm}
Let $X$ be a nonreduced scheme, and let $t: X_\red \to X$ be the inclusion of the associated reduced scheme.  Then $t_*: \cM(X_\red) \to \cM(X)$ is an equivalence of categories.
\end{thm}

\section{Changes in Conventions and Notation}
\label{sect:litguide}

Staggered sheaves were introduced in~\cite{a:sts} in a much more general setting than we have considered here: $G$ and $X$ were schemes of finite type over a noetherian base scheme admitting a dualizing complex, with no assumption on orbits.  With hindsight, this setting appears to have been too general: the key results really do require $X$ to be a variety with finitely many orbits.  The latter setting permits great simplifications of the basic definitions, and thus the definitions in the present note may seem to bear little resemblance to their counterparts in~\cite{a:sts} and subsequent papers.  In this appendix, we briefly indicate how to connect the notions in this note to those in the various papers.

\subsection{$s$-structures}
\label{subsect:sstruc}

In~\cite{a:sts}, an \emph{$s$-structure} on $X$ is a collection of full
subcategories $(\{\cgl Xw\}, \{\cgg Xw\})_{w \in \Z}$ of $\cg X$,
satisfying a rather lengthy list of axioms.  However, on a variety with
finitely many orbits, the ``gluing theorems''~\cite[Theorem~5.3]{a:sts}
and~\cite[Theorem~2.1]{as:sspfv} allow one to reduce the task of specifying
an $s$-structure to specifying one on each orbit, and by~\cite[\S
7]{at:pdtss}, the latter is equivalent to giving a certain central
cocharacter of the isotropy group.

The definition in~\eqref{it:recsplit} actually corresponds to what
have been called ``recessed, split $s$-structures'' in~\cite{at:bstc,
at:pdtss, a:qhpss}. The main theorems all require recessed, split
$s$-structures, so in the present note, that terminology has been dropped,
and those conditions incorporated into the definition of ``$s$-structure.''

\subsection{Perversities}
\label{subsect:perv}

In~\cite{a:sts}, a perversity function was required to obey certain
inequalities, known as the ``monotone'' and ``comonotone'' conditions, in
order for Theorem~\ref{thm:tstruc} to hold, and
Theorem~\ref{thm:simple} was proved only for perversities obeying even
stronger inequalities (``strictly monotone and comonotone'').  These
restrictions were removed in~\cite{at:bstc} by entirely new proofs using
baric structures.  (These new proofs assume ``recessed'' $s$-structures,
however.)  Thus, a perversity function is now an arbitrary function $\orb X
\to \Z$.

\subsection{Duality and codimension}
\label{subsect:duality}

In this note, the Serre--Grothendieck duality functor $\sgD$ has been
defined in terms of ``the'' dualizing complex $\omega_X$, using the
assumption~\eqref{it:embed} to guarantee its existence in positive
characteristic.  Let us now take \emph{dualizing complex} to mean any
object $\tilde\omega_X$ with the property that the functor $\cRHom(\cdot,
\tilde\omega_X)$ is an antiautoequivalence of $\Db X$.  Such dualizing
complexes exist in very great generality (see~\cite[\S V.10]{har:rd}
and~\cite[Proposition~1]{bez:pcs}), without any assumption
like~\eqref{it:embed}.  In the papers~\cite{a:sts, at:bstc, at:pdtss}, the
set-up included the \emph{choice} of a dualizing complex.  No results on
staggered sheaves depend in a substantive way on this choice, but various
formulas such as~\eqref{eqn:dualr} must be modified.

Specifically, for any orbit $C$, the object $i_C^!\tilde\omega_X$ must be a shift of a line bundle; say $i_C^!\tilde\omega \simeq \cL_C[n_C]$.  The \emph{altitude} of $C$, denoted $\alt C$, is defined to be $\step \cL_C$.  We also put $\cod C = -n_C$.  The assumptions of the present note give $\alt C = \step \cano C$ and $\cod C = -\dim C$.  The latter assumption was also made in~\cite{a:qhpss}.

\subsection{Baric and co-$t$-structures}
\label{subsect:barcot}

The definition of $\q\Dsl Xw$ used in
Appendix~\ref{sect:baric} follows~\cite{at:bstc}, but the definition used
in~\cite{at:pdtss} contains an extra factor of $2$, for reasons explained
therein.  Co-$t$-structures were introduced
in~\cite{at:pdtss} in a way that depended on an additional function $q:
\orb X \to \Z$, and the associated categories were denoted $\q\Dkl Xw$ and $\q\Dkg Xw$.  These coincide with the $\ru\Dkl Xw$ and $\ru\Dkg Xw$ of the present note when $q = \llcorner r \lrcorner$, cf. Proposition~\ref{prop:cohom}.


\end{document}